\documentclass[12pt,a4paper]{amsart}

\usepackage[utf8]{inputenc}
\usepackage[T1]{fontenc}
\usepackage{lmodern}
\usepackage{mathtools}
\usepackage{standard}
\usepackage{amsrefs}
\usepackage{enumerate}
\usepackage[usenames,dvipsnames]{color}
\usepackage[colorlinks=true,linkcolor=Red,citecolor=Green]{hyperref}

\numberwithin{theorem}{section}

\newcommand{\rd}{F}
\newcommand{\trd}{\widetilde{F}}
\DeclareMathOperator{\Tr}{Tr}

\DeclareMathOperator{\F}{\mathcal{F}}
\DeclareMathOperator{\sgn}{sgn}
\DeclareMathOperator{\crit}{crit}
\newcommand{\CP}{\mathbb{CP}}
\newcommand{\Sph}{\mathbb{S}}
\newcommand{\symb}{\Gamma}
\newcommand{\symbcl}{\Gamma_{\mathrm{cl}}}
\newcommand{\calc}{G}
\newcommand{\calccl}{G_{\mathrm{cl}}}

\newcommand{\hamvf}{\mathsf{H}}
\newcommand{\schwartz}{\mathcal{S}}

\newcommand{\Xray}{\mathsf{X}}

\newcommand{\barg}{\mathcal{B}}
\newcommand{\phase}{\varphi}
\newcommand{\weight}{\Phi}
\newcommand{\trafo}{\kappa_\phase}
\newcommand{\Lagr}{\Lambda_\weight}
\newcommand{\Opw}{\Op^w}

\DeclareMathOperator{\Vol}{\mathrm{vol}}
\DeclareMathOperator{\ad}{ad}
\DeclareMathOperator{\Ad}{Ad}

\newcommand{\kk}{\left( \frac{k}{2\pi} \right)}

\newcommand{\kahler}{K\"ahler }

\newcommand{\R}{\mathbb{R}}
\newcommand{\C}{\mathbb{C}}
\newcommand{\Z}{\mathbb{Z}}

\newcommand{\hcal}{\mathcal{H}}

\newcommand{\ncal}{\mathcal{N}}
\newcommand{\ocal}{\mathcal{O}}

\newcommand{\half}{\frac{1}{2}}

\newcommand{\h}{\hat} 

\title[Schr\"odinger Trace Invariants]{Schr\"odinger Trace Invariants for Homogeneous Perturbations of the Harmonic Oscillator}
\date{\today}

\author[M. Doll]{Moritz Doll}
\address{Department 3 -- Mathematics, University of Bremen, Bibliotheksstr. 1, \newline\indent D-28359 Bremen, Germany}
\email{doll[AT]uni-bremen.de}

\author[S. Zelditch]{Steve Zelditch}
\address{Department of Mathematics, Northwestern University, 2033 Sheridan Rd., \newline\indent Evanston IL 60208, USA}
\email{zelditch[AT]math.northwestern.edu}

\thanks{
We would like to thank Jared Wunsch for many useful discussions and inviting the first author to Northwestern University, where the project was started. We would also like to thank Peng Zhou for many discussions related
to the Bargmann-Fock setting. 
The second author gratefully acknowledge the support of NSF grant DMS-1810747.}

\begin{document}
\begin{abstract}
    Let $H = H_0 + P$ denote the harmonic oscillator on $\mathbb{R}^d$ perturbed by an isotropic pseudodifferential operator $P$ of order $1$ and let
    $U(t) = \operatorname{exp}(- it H)$. We prove a Gutzwiller-Duistermaat-Guillemin
    type trace formula for $\operatorname{Tr} U(t).$ The singularities occur at times
    $t \in 2 \pi \mathbb{Z}$ and the coefficients involve the dynamics of the Hamilton flow of the symbol
    $\sigma(P)$ on the space $\mathbb{CP}^{d-1}$ of harmonic oscillator orbits of
    energy $1$. This is a novel kind of sub-principal symbol  effect on the trace. We generalize the averaging technique of Weinstein and Guillemin to this
    order of perturbation, and then
    present two completely  different calculations of  $\operatorname{Tr} U(t)$.
    The first proof directly constructs a parametrix of $U(t)$ in the isotropic calculus, following earlier work of Doll--Gannot--Wunsch.
    The second proof conjugates the trace  to the Bargmann-Fock setting, the order $1$ of the perturbation coincides with  the  `central limit scaling' studied by Zelditch--Zhou for Toeplitz operators.
\end{abstract}

\maketitle

\section{Introduction}
In a recent article \cite{DGW}, a two-term Weyl law with remainder was proved for special types of perturbations $H = H_0 + P$ of 
 the homogeneous isotropic harmonic oscillator
\begin{align*}
    H_0 = (1/2) (-\Lap + |x|^2)
\end{align*}
with symbol $p_2 = (1/2) (|x|^2 + |\xi|^2)$.
The perturbation $P$   is a classical isotropic pseudodifferential operator of order 1, $P = \Opw(p) \in \calccl^1$ with isotropic symbol $p \in \symbcl^1$.
The salient feature of the perturbation   is that   its  order is at a threshold or critical level,  so that  the perturbation has a strong contribution to the  Weyl law.
The purpose of this note is to go further in the spectral analysis, by considering the Duistermaat--Guillemin--Gutzwiller trace formula for the propagator $U(t) = e^{-itH}$. It was shown 
in \cite{DGW} that the distribution trace $\Tr U(t)$ has singularities 
at the same points $t \in 2 \pi \Z$ as for the unperturbed propagator. We
give a new proof of this result and, more importantly, determine the singularity coefficients
at non-zero singular times $t = 2 \pi k$. It turns out that the coefficients
involve the Hamiltonian mechanics of $p$ on the space
$\CP^{d-1}$ of Hamilton orbits of $H_0$. After generalizing the averaging technique of
Weinstein \cite{Weinstein77} and Guillemin \cite{Guillemin81}, we  prove the result in two
ways:
(i) using special parametrices for the isotropic propagators $U(t)$
on $L^2(\R^d)$  developed in \cite{DGW}, and (ii) by using an eigespace
decomposition for $H_0$ and  conjugating the traces of $U(t)$ on eigenspaces to  the Bargmann--Fock representation and then to holomorphic sections of line bundles over the space $\CP^{d-1}$ of Hamilton orbits of $H_0$. It turns out that the eigenspace traces of $U(t)$ involve the  central limit scaling of Toeplitz propagators studied in    \cites{ZeZh17,ZeZh18-1}. We then use  the trace asymptotics of \cite{ZeZh18-1} 
to obtain a second proof.
The two approaches are in some sense dual,
the first  leading to Fourier integral representations, and the second to Fourier series representations, of $\Tr U(t)$.
The second approach shows (see Corollary \ref{MAINCOR}) that $\Tr U(t)$ trace is a  sum of distributions on $\R$ of the form $e^{i t a |D|^{\half}} e_r(x)$
where $D = \frac{1}{i} \frac{d}{dt}$, $a \in \R$ and $e_r(t)$ is a periodic classical homogeneous distribution of order $r$.
The operator $e^{i  a t|D|^{\half}}$ is a pseudo-differential operator on $\R$ with symbol $e^{ it a \sqrt{|\xi|}}$ in the class $S^0_{\half, 0}$, i.e. each derivative decreases the order by $\half$.
The proofs actually show that the unusual contribution of $P$ to the wave trace occurs whenever one perturbs an operator  $H_0$ with periodic bicharacteristic flow and maximally high multiplicities by an operator of `one-half lower order'.
See Section \ref{RELATED}.\footnote{The order convention in this article is that $H_0$ has order $2$ and $P$ has order $1$. But it is also natural to define the order of $H_0$ to be $1$
and that of $P$ to be $\frac{1}{2}$ and this order relation generalizes to many other settings.}

To state the results, we introduce some notation.
Let $\hamvf_0$ denote the Hamiltonian vector field of $p_2=(1/2)(|x|^2+|\xi|^2)$, whose flow $(x(t),\xi(t)) = \exp(t \hamvf_0)(x_0,\xi_0)$ satisfies
\begin{align*}
	x(t) &= \cos(t)x_0 + \sin(t)\xi_0,\\
	\xi(t) &= \cos(t)\xi_0 - \sin(t)x_0.
\end{align*}
We define
\begin{equation} \label{eq:Xray}
    \Xray f(x,\xi)=\frac{1}{2\pi}\int_0^{2\pi} f(\exp(t\hamvf_0)(x,\xi))\, dt.
\end{equation}
For $f \in \CI(\Sph^{2d-1}) = \{p_2 = 1\}$ we use the same notation $\Xray f$.\footnote{Note that the definition differs from the one in \cite{DGW} by a factor of $2\pi$.}
The map $\Xray$ induces a map
\begin{align*}
    \tilde{\Xray} : \CI(\RR^{2d}) &\to \CI(\CP^{d-1}),
\end{align*}
where $\CP^{d-1} = \{p_2 = 1\} / \sim$ and $z_1 \sim z_2$ if and only if there exists a $\theta \in \RR$ such that $z_1 = e^{i\theta} z_2$.
Note, that we can write $\tilde\Xray$ as
\begin{align*}
    \tilde{\Xray}f(z) = \frac{1}{2\pi} \int_0^{2\pi} f(e^{i\theta}z) d\theta
\end{align*}
using the natural identification $\CC^d = \RR^{2d}$.

In both approaches we critically use an averaging method originally due to Weinstein~\cite{Weinstein77} to simplify the perturbation (cf. Lemma~\ref{lem:average}).
If $p_1 = \sigma^1(P)$, then the principal symbol of the averaged perturbation can be identified with $\tilde{\Xray} p_1 : \CP^{d-1} \to \RR$.
Assume that $\tilde\Xray p_1$ is a Morse function on $\CP^{d-1}$.
The set of stationary points is given by $\Pi_{2\pi} = \{z_j\}_{j=1}^n = \{ z \in \CP^{d-1} \colon d\, \tilde\Xray p_1(z) = 0\}$.
Furthermore, we set $d_j = \abs{\det D d \tilde\Xray p_1(z_j)} \not = 0$ and $\sigma_j$ the signature of the Hessian, $\sgn Dd\tilde\Xray p_1(z_j)$.

\subsection{Statement of Results in the Homogeneous Isotropic Calculus}
It was shown in \cite{DGW}*{Theorem 1.2} that under the assumption that $\tilde\Xray p_1$ is a Morse function on $\CP^{d-1}$, the singularity at $2\pi k$ is of order
$O(\lambda^{(d-1)/2})$.

The content of the main theorem is to calculate the principal term of the distribution at $t = 2\pi k$.
\begin{theorem}\label{thm:principal-invariant}
    Let $H$ be as above and $w(t) = \Tr e^{-itH}$ its Schrödinger trace.
    For $t \in (-\ep,\ep)$ and $k \in \NN$,
    \begin{align*}
        \F^{-1}\{w(t - 2\pi k)\}(\lambda) \sim \lambda^{(d-1)/2}e^{2\pi k i \lambda} \sum_{j=1}^n e^{-2\pi ik\lambda^{1/2}\tilde\Xray p_1(z_j)} \sum_{l=0}^\infty \lambda^{-l/2}\gamma_{k,j,l}
    \end{align*}
    and \[\gamma_{k,j,0} = (\pi k)^{-(d-1)} d_j^{-1/2} e^{\pi i (-\sigma_j/4 + dk)} e^{-2\pi i k \tilde\Xray p_0(z_j)}.\]
\end{theorem}

\subsection{\label{BFINTRO}Statement of Results in the Bargmann-Fock Setting}
We give a different proof of Theorem~\ref{thm:principal-invariant} in which we first  expand $\Tr U(t)$ 
into traces on eigenspaces of $H_0$ and then  conjugate by the Bargmann transform to obtain traces of semi-classical Toeplitz operators 
in the Bargmann-Fock setting. The advantage of the conjugation is
that the spectral projections for $H_0$ conjugate to the simpler Bergman kernel projections for holomorphic sections of the standard ample line bundles over projective space.
In that setting, Theorem \ref{thm:principal-invariant} can be reduced to the trace asymptotics  calculated in \cite{ZeZh18-1} for semi-classial Toeplitz operators acting on holomorphic sections of line bundles over \kahler manifolds.

Let $L^2(\R^d)  = \bigoplus_{N=0}^{\infty} \hcal_N$ be the decomposition of $L^2(\R^d)$ into eigenspaces of $H_0$, and let 
\begin{equation} \label{PiN}
    \Pi_N: L^2(\R^d) \to \hcal_N
\end{equation}
be the orthogonal projection, i.e. the spectral projections for $H_0$.
By the averaging Lemma \ref{lem:average-smoothing},  $H_0 + P$ is unitarily equivalent to $ H_0 + B + R$ where $[B, H_0] = 0$ and $R : \schwartz'(\RR^d) \to \schwartz(\RR^d)$.
We can assume that that $R = 0$ (cf. Corollary~\ref{cor:smoothing}).
Using the  known spectrum of $H_0$, we have the following:

\begin{lemma}\label{FULLTRACE}
    The trace  $\Tr e^{it H}$ is given by
    \begin{align*}
         \Tr e^{i t (H_0 + P)} = \Tr e^{it(H_0 + B)} = \sum_{N=0}^{\infty} \Tr e^{i t B} |_{\hcal_N} \; e^{i t (N + \frac{d}{2})}.
    \end{align*}
\end{lemma}
It therefore suffices to determine the asymptotics as $N \to \infty$ of
\begin{equation} \label{BTRACE}
    \Tr e^{i t B}|_{\hcal_N} =  \Tr  \Pi_N e^{i t B}\Pi_N
\end{equation}
for each $t$ and in particular for $t = 2 \pi k$ for some $k \in \Z$. 
Using the trace asymptotics of \cite{ZeZh18-1}, we prove (in the notation
of Theorem \ref{thm:principal-invariant}):

\begin{theorem} \label{INTROTHEO}
    Let $N > 0$,  $m = d-1$ and $t = 2 \pi k$. Then, $\Tr \Pi_N e^{itB}$ admits a complete asymptotic expansion as $N \to \infty$ in powers of $(t \sqrt{N})$ with leading term,
    \begin{align*}
        \Tr \Pi_N e^{itB} &= \left( \frac{N}{2\pi} \right)^m \left(\frac{t \sqrt{N}}{4 \pi}\right)^{-m} \sum_{j=1}^n d_j^{-1/2} \cdot e^{it\sqrt{N} \tilde{\Xray}p_1(z_j)}e^{i\pi \sigma_j/4} (1 + O(|t|^3 N^{-1/2})).
    \end{align*}
\end{theorem}

To prove Theorem \ref{INTROTHEO} we relate the unitary group $\Pi_N e^{it B} \Pi_N$ to semi-classical Toeplitz Fourier integral operators.
Since $B$ has order $1$, it does not generate a semi-classical Toeplitz Fourier integral  operator. Rather, $V_N(t) = \Pi_N e^{it \sqrt{N} B} \Pi_N$ 
is a semi-classical Toeplitz Fourier integral  operator.
We then employ the trick that $\Pi_N e^{it B} \Pi_N = V_N(\frac{t}{\sqrt{N}})$ to express  $\Pi_N e^{it B} \Pi_N $ as a time-scaled Toeplitz Fourier integral unitary group. 
Under conjugation by the Bargmann transform, $V_N(t)$ is carried to a Toeplitz Fourier integral  operator $U_N(t)$ in the complex domain. Combining with
Lemma~\ref{FULLTRACE} gives, 

\[\Tr e^{i t (H_0 + P)} = \sum_{N=0}^{\infty} \Tr U_N(\frac{t}{\sqrt{N}}) \; e^{i t (N + \frac{d}{2})}. \]
The time-scaled Toeplitz Fourier integral operators $ U_N(\frac{t}{\sqrt{N}})$ are studied in \cites{ZeZh17, ZeZh18-1} for unrelated reasons,
and the pointwise and integrated asymptotics of those articles gives Theorem \ref{INTROTHEO}.
 
We then substitute the asympotics of Theorem \ref{INTROTHEO} into
Lemma \ref{FULLTRACE}. When $B = 0$,  
$\Tr e^{-itH}$ is a sum of Hardy distributions of the form,\footnote{
By a Hardy distribution is meant one with only positive frequencies}
 
\begin{equation} \label{UNPTR}
    \mu_{N,d,r}(t) \coloneqq \sum_{N=0}^{\infty} N^r   e^{it (N+ \frac{d}{2}) },\; (r \in \half \Z),
\end{equation}
 with    $r = d-1$.
When $B \not= 0$ and $\sigma_B$ descends to a Morse function on $\CP^{d-1}$,  the order is $\frac{d-1}{2}$ and the fact $e^{i t \sqrt{k} H(z_c)}$ changes the singularity
type to  a sum of distributions of the type
\[e_{N,d,r} (t,a) = \sum_{N = 1}^{\infty}    N^{r}  e^{it (N + \sqrt{N} a + \frac{d}{2}) }, \;\; (a \in \R).\]
These are non-standard and apparently novel types of homogeneous Lagrangian distributions (the authors have not found them in prior articles).
We  obtain an interesting description of $\Tr e^{- it H}$ and
its singularities in terms of these distributions,

\begin{corollary} \label{MAINCOR}
    The leading order singularity in $t$ of $\Tr e^{i t (H_0 + B)} $
    at $t = 2 \pi k$ is the
    same as the leading order singularity of the finite sum,
    \[ \sum_{j=1}^n  \frac{e^{i \pi \sigma_j/4}}{d_j^{1/2}} e_{N,d, (d-1)/2} (t, \tilde\Xray p_1).\]
    This explict expansion of the trace in terms of non-standard Lagrangian
    distributions does not seem to follow easily from the calculation in the
    Schroedinger representation.
\end{corollary}
We further  observe that  $ e^{i t a \sqrt{|D|}} 
\in \Psi^0_{\half, 0}(\R)$, i.e. is a pseudo-differential operator on $\R$
of order zero, with symbol $e^{it a \sqrt{|\xi|}}$. Each derivative decreases
the order of the symbol of by $\half$.  
It follows that although $e_{N,d}(t,a, (d-1)/2)$ is not a homogeneous distribution,
it has the same properties as a homogeneous distribution. For instance,
since $e^{it a \sqrt{|D|}}$ is pseudo-local, $e_{N,d}(t,a)$ has singularities
at the same points $t \in 2 \pi \ZZ$ as $e_{N,d}(t,0)$, as first proved
in \cite{DGW}.

To relate the expansions of Theorem \ref{thm:principal-invariant} and Theorem \ref{INTROTHEO},
it is  proved in Section~\ref{EQUIVSECT}
that the asymptotics $\lambda \to \infty$ in Theorem~\ref{thm:principal-invariant} and the asymptotics $N \to \infty$ in Theorem~\ref{INTROTHEO} agree to leading order.

\subsection{\label{RELATED}Related Problems} The trace asymptotics for the propagator
$e^{-i t H}$ of the perturbation $H = H_0 + P$ has many analogues in
settings where one has a first order elliptic pseudo-differential  operator $H_0$ with periodic bicharacteristic
flow and maximally high multiplicities and $P$ is a  $\half$-order below that of $H_0$. Without trying to state the most general result, we note that it is
valid
for $H_0 = \sqrt{-\Delta}$ on $S^n$, the standard sphere, and $P$ is a pseudo-differential operator of order $\half$. Indeed, as in
the proof of Theorem \ref{BTRACE},  it suffices to work out the semi-classical trace asymptotics in the eigenspaces, which only uses that $U_N(t)$ is a semi-classical unitary Fourier integral operator.   The main results pertain to the special effect of perturbation of the special order $\half$ and hold
for such perturbations on all rank one symmetric spaces. The results can also be adapted to general Zoll manifolds if $\sqrt{-\Delta}$ is replaced
by an operator $\ncal$ which indexes the `bands' in the sense of \cite{Guillemin84}.

\section{Isotropic Calculus and Quantization}
The class of isotropic symbols of order $m \in \RR$, $\symb^m(\RR^d)$, consists of the smooth functions $a \in \CI(\RR^{2d})$ such that
\begin{align*}
    |\pa^\alpha_x \pa^\beta_\xi a(x,\xi)| \lesssim_{\alpha,\beta} \ang{(x,\xi)}^{m-(|\alpha|+|\beta|)}, \quad \text{for all } \alpha,\beta \in \NN^d.
\end{align*}
We will mainly be interested in the subclass of \emph{classical} symbols $a \in \symbcl^m(\RR^d)$, which admit an asymptotic expansion in homogeneous terms of order $m-j$.
To any isotropic symbol $a \in \symb^m(\RR^d)$ we associate an isotropic pseudodifferential operator $A \in \calc^m(\RR^d)$ by the \emph{Weyl-quantization},
\begin{align*}
    A = \Opw(a) = (2\pi)^{-d}\int e^{i(x-y)\xi} a( (x+y)/2, \xi) d\xi.
\end{align*}
The integral is defined as an oscillatory integral and we further note that any bounded linear operator $A : \schwartz(\RR^d) \to \schwartz'(\RR^d)$ has a Weyl-symbol $a \in \schwartz'(\RR^{2d})$, which satisfies $\Opw(a) = A$.
This follows directly from the Fourier inversion formula and the Schwartz kernel theorem. We  also use  that the class of residual isotropic operators, $\calc^{-\infty}(\RR^d) = \bigcap_{m \in \RR} \calc^m(\RR^d)$ is the space of \emph{(globally) smoothing} operators $A : \schwartz'(\RR^d) \to \schwartz(\RR^d)$.
For the main properties of the isotropic calculus, we refer to \cite{DGW} (see also Helffer~\cite{Helffer84} and Shubin~\cite{Shubin78}).

For later reference, we recall the formula for composing isotropic pseudodifferential operators.

\begin{proposition}\label{prop:composition}
    Let $a \in \symbcl^{m_1}, b \in \symbcl^{m_2}$ two isotropic symbols.
    Then $\Opw(a)\Opw(b) = \Opw(c)$, where
    \begin{align*}
        c(x,\xi) = (a \# b)(x,\xi) = e^{iA(D)} a(x,\xi) b(y,\eta)\big|_{y=x,\eta=\xi}.
    \end{align*}
    The exponential is defined as an Fourier multiplier and the operator $A(D)$ is given by $A(D) = (\ang{D_\xi,D_y} - \ang{D_x,D_\eta})/2$.
    Moreover, there exists an asymptotic expansion
    \begin{align*}
        (a \# b)(x,\xi) \sim \sum_k \frac{i^k}{k!} A(D)^k a(x,\xi) b(y,\eta)\big|_{y=x,\eta=\xi}.
    \end{align*}
\end{proposition}

\section{Averaging}
In this section we prove that there exists a $B \in \calccl^1(\RR^d)$ and $R \in \calc^{-\infty}(\RR^d)$ such that
\begin{align*}
    H \cong H_0 + B + R, \quad [H_0, B] = 0
\end{align*}
and calculate its symbol. We will follow the arguments of \cites{Weinstein77,Guillemin81}.

For the prove of Theorem~\ref{thm:principal-invariant}, the following lemma suffices.
\begin{lemma}\label{lem:average}
    Let $H = H_0 + P$ with $P \in \calccl^1(\RR^d)$. For any $N \in \NN$ there exists a unique $B_{-N} \in \calccl^1(\RR^d)$ modulo $\calccl^{-N}(\RR^d)$
    and an unitary operator $U_N$ such that
    \begin{align*}
        [H_0, B_{-N}] &\in \calccl^{-N}(\RR^d),\\
        U_N^{-1} H U_N &= H_0 + B_{-N}.
    \end{align*}
    Furthermore, for $N > 1$ the Weyl quantized symbol of $B_N$ has an asymptotic expansion
    \begin{align*}
        b = \Xray p + \symbcl^{-1}(\RR^d)
    \end{align*}
    where $p$ is the Weyl-quantized symbol of $P$.
\end{lemma}
\begin{proof}
    For any $A = a^w(x,D) \in \calccl^m$, we define $A(t) = e^{itH_0} A e^{-itH_0}$ and its Weyl-quantized symbol is given by
    \begin{align*}
        a(t) = a \circ \exp(t\hamvf_0).
    \end{align*}
    The average of $A$ is given by $(1/2\pi) \int_0^{2\pi} A(t) dt$ and $(1/2\pi) \int_0^{2\pi}\int_0^t A(s) \,ds\,dt$.

    First, we define for $F \in \calc^k$ the operator $\ad(F) : \calc^m \to \calc^{m+k-2}$ by
    \begin{align*}
        \ad(F)A = [F,A].
    \end{align*}
    If $k < 2$ we obtain, using the Taylor expansion of the exponential, that
    \begin{align*}
        e^{-i \ad(F)} A = A -i[F,A] + \calc^{m-2(2-k)}.
    \end{align*}
    We define the pseudodifferential operator $F_1$ by
    \begin{align*}
        F_1 = \frac{1}{2\pi} \int_0^{2\pi}\int_0^t U_0(-s) P U_0(s) \,ds\,dt.
    \end{align*}
    It has the property that
    \begin{align*}
        [F_1, H_0] = P - B_1,
    \end{align*}
    where
    \begin{align*}
        B_1 = \frac{1}{2\pi} \int_0^{2\pi} U_0(-t)PU_0(t) \,dt.
    \end{align*}
    Therefore, we calculate
    \begin{align*}
        e^{-i \ad(F_1)} H = H_0 + B_1 + [F_1, B_1] + \calccl^{-1}.
    \end{align*}
    Now, we inductively lower the order of the remainder term.
    Set $R_0 = [F_1, B_1] \in \calc^0$ then define $F_0$ and $B_0$ as above, but with $P$ replaced by $R_0$.
    Then we find that
    \begin{align*}
        e^{-i \ad(F_0)} e^{-i \ad(F_1)} H = H_0 + B_1 + B_0 + \calccl^{-1},
    \end{align*}
    thus concluding
    \begin{align*}
        e^{-i \ad(F_{-N})} \dotsm e^{-i \ad(F_1)} H = H_0 + B_1 + \dotsc + B_N + \calc^{-N-1}
    \end{align*}
    with $[H_0, B_1 + \dotsc + B_N] = 0$.
    The operator $U_N$ is given by $e^{iF_{-N}} \dotsm e^{iF_1}$. This follows from the fact that
    \begin{align*}
        \Ad_{e^{-iF_j}} = e^{-i\ad(F_j)},
    \end{align*}
    where $\Ad_U A = U A U^{-1}$.

    The Weyl-quantized symbol of $B_1$ is given by
    \begin{align*}
        b_1 = \Xray p = \Xray p_1 + \Xray p_0 + \symbcl^{-1}.
    \end{align*}
    We show that $B_0 \in \calccl^{-1}(\RR^d)$. For this, observe that
    \begin{align*}
        B_0 &= \frac{1}{2\pi} \int_0^{2\pi} U_0(-t) [F_1, B_1] U_0(t) \,dt\\
        &= \frac{1}{2\pi} \int_0^{2\pi} [U_0(-t) F_1 U_0(t), B_1] \,dt\\
        &= \left[\,\frac{1}{2\pi} \int_0^{2\pi} U_0(-t) F_1 U_0(t)\,dt, B_1\right].
    \end{align*}
    Here, we have used that $B_1$ commutates with $U(t)$.
    By Fubini, we calculate the principal symbol of the left entry in the commutator,
    \begin{align*}
        \frac{1}{(2\pi)^2} \int_0^{2\pi} \int_0^{2\pi}\int_0^t p_1 \circ \exp( (s+t')\hamvf_0) \,ds\,dt\,dt' &= \frac{1}{2\pi} \int_0^{2\pi}\int_0^t \Xray p_1\,ds\,dt\\
        &= \pi \cdot \Xray p_1.
    \end{align*}
    The second equality follows from the fact that $\Xray p_1$ is constant along the Hamiltonian flow and that $\int_0^{2\pi} \int_0^t ds\,\,dt = 2\pi^2$.
    Now, the principal symbol of $U_0(-t) F_1 U_0(t)$ and $B_1$ are the same up to a constant, therefore its commutator has zero principal symbol.
\end{proof}
Now, we have to make sure that we can use asymptotic summation.
\begin{lemma}\label{lem:average-smoothing}
    Let $H \in \calccl^2(\RR^d)$ as above. There exists a unitary pseudodifferential operator $U \in \calccl^0(\RR^d)$ and a self-adjoint $B \in \calccl^1(\RR^d)$ such that
    \begin{align*}
        [H_0, B] &= 0,\\
        U^* H U &= H_0 + B + \calc^{-\infty}.
    \end{align*}
\end{lemma}

\begin{proof}
    Let $U \in \calccl^0(\RR^d)$ and $B \in \calccl^1(\RR^d)$ such that
    \begin{align*}
        \tilde{U} - U_k &\in \calccl^{-(k+1)}(\RR^d)\\
        B - B_k &\in \calccl^{-k}(\RR^d)
    \end{align*}
    for all $k \in \NN$. Here, $U_k$ and $B_k$ are as constructed in the proof of the previous lemma. Consider the bounded self-adjoint operator $F = \tilde{U}\tilde{U}^* - \id$. Since $U_k$ is unitary for all $k$, it follows that
    $F \in \calc^{-\infty}$ and hence compact. Thus, we may assume that its $L^2$-norm is bounded by $C \in (0,1)$, by modifying it on a finite-dimensional subspace.
    We let $K = \sum_{j=1}^\infty c_j F^j$, where $c_j$ are the Taylor coefficients of the expansion of $(1 + t)^{-1/2}$ at the origin. The operator $K$ is well-defined since the series converges for $|t|< 1$ and smoothing.
    Putting $U = (\id+K)\tilde{U}$, we see that
    \begin{align*}
        UU^* &= (\id + K) \tilde{U}\tilde{U}^* (\id + K)\\
        &= (\id + K)^2 (\id + F)\\
        &= \id.
    \end{align*}
    We can apply the averaging to $B$ again so that $B$ commutes with $H_0$ and $U^* H U = H_0 + B + \calc^{-\infty}$ as claimed.
\end{proof}

It follows from Duhamel's formula (see Proposition \ref{PROPA1} of the Appendix) that
one has,

\begin{corollary} \label{cor:smoothing}
    With the same notation as in Lemma \ref{lem:average-smoothing},
    \[ \Tr e^{- i t H} = \Tr e^{- i t (H_0 + B)} \quad \text{ mod }\; \CI(\R).\]
\end{corollary}

\section{Fourier Integral Operators}
In what follows, we will encounter operators of the form $A = \Opw(e^{i\phi}a)$, where $\phi$ is homogeneous of degree one outside a compact set in $\RR^{2d}$.
These operators are clearly not isotropic pseudodifferential operators, because they are not pseudolocal.
Since spatial derivatives of isotropic symbols gain decay in both $x$ and $\xi$, we can still compose $A$ with and any isotropic pseudodifferential operator and explicitly calculate the asymptotic expansion.
\begin{proposition}\label{prop:composition-fio}
    Let $p \in \symbcl^m$, $a \in \symbcl^0$, and $\phi \in \symbcl^1$.
    Then
    \begin{align*}
        p \# e^{i\phi} a = e^{i\phi} c,
    \end{align*}
    where
    \begin{align*}
        c_m &= p_m a_0,\\
        c_{m-1} &= p_m a_{-1} + p_{m-1} a_0 + (1/2)\{p_m, \phi_1\}a_0,\\
        c_{m-2} &= p_m a_{-2} + p_{m-1} a_{-1} + p_{m-2} a_0\\
        &\phantom{=} + (1/2) \left(\{p_{m-1}, \phi_1\}a_0 + \{p_m,\phi_1\}a_{-1} - i\{p_m,a_0\}\right)\\
        &\phantom{=} + (1/4) \left(\sum_{j,k}\pa_{x_jx_k}p_m \pa_{\xi_j}\phi_1 \pa_{\xi_k}\phi_1 - \pa_{\xi_j\xi_k}p_m \pa_{x_j}\phi_1\pa_{x_k}\phi_1\right)a_0.
    \end{align*}
\end{proposition}
\begin{proof}
    The asymptotic expansion follows from the composition theorem, Proposition~\ref{prop:composition}, and ordering the terms by homogeneity.
\end{proof}

Moreover, we need a composition result for quadratic phase functions:
\begin{proposition}
    Let $A \in \RR^{d\times d}$ be symmetric and $a, b\in \schwartz(\RR^d)$. There is an integral representation
    \begin{align*}
        (e^{i\ang{A\cdot, \cdot}} a \# b)(z) = \pi^{-2d} \int_{\RR^{4d}} e^{-2i \ang{Q w, w}} a(z+w_1) b(z + w_2 + (1/2)JA (w_1+z)) dw_1 \,dw_2,
    \end{align*}
    where
    \begin{align*}
        Q = \begin{pmatrix} A/2 & -J \\ J & 0\end{pmatrix}.
    \end{align*}
\end{proposition}
\begin{proof}
    By Zworski~\cite{Zworski12}*{Theorem 4.11} we have the integral representation
    \begin{align*}
        (e^{i\ang{A\cdot, \cdot}} a \# b)(z) = \pi^{-2d} \int_{\RR^{4d}} e^{-2i \sigma(w_1,w_2)} e^{i\ang{A(z+w_1), z+w_1}} a(z+w_1) b(z+w_2) \,dw_1\,dw_2.
    \end{align*}
    Define the phase function
    \begin{align*}
        \Phi(w_1, w_2, z) &= -2 \sigma(w_1, w_2) + \ang{A(z+w_1),z+w_1}\\
        &= -2 \ang{\begin{pmatrix} 0 & -J \\ J & 0\end{pmatrix} \begin{pmatrix}w_1 \\w_2 \end{pmatrix}, \begin{pmatrix}w_1 \\w_2 \end{pmatrix}} + \ang{A (z+w_1), z+w_1}.
    \end{align*}
    Changing coordinates
    \begin{align*}
        \tilde w_1 &= w_1,\\
        \tilde w_2 &= w_2 - (1/2)J A (w_1 + z)
    \end{align*}
    yields
    \begin{align*}
        \Phi(w_1,w_2,z) = \tilde{\Phi}(\tilde w_1, \tilde w_2, z) \coloneqq -2\ang{ \begin{pmatrix} A/2 & -J \\ J & 0\end{pmatrix} \begin{pmatrix} \tilde w_1 \\ \tilde w_2\end{pmatrix}, \begin{pmatrix} \tilde w_1 \\ \tilde w_2\end{pmatrix}} + \ang{Az,z}.
    \end{align*}
    Hence, we have
    \begin{align*}
        (e^{i\ang{A\cdot, \cdot}/2} a \# b)(z) &= \pi^{-2d} \int_{\RR^{4d}} e^{i \tilde{\Phi}(w_1,w_2)} a(z+w_1) b(z + w_2 + (1/2)JA(w_1 + z)) dw_1 \,dw_2\\
        &= \pi^{-2d} e^{i\ang{Az,z}}\\
        &\phantom{=} \cdot \int_{\RR^{4d}} e^{-2i \ang{Qw,w}} a(z+w_1) b(z + w_2 + (1/2)JA(w_1 + z)) dw_1 \,dw_2.
    \end{align*}
\end{proof}
We also have to calculate how quadratic exponentals act on oscillating functions:
\begin{proposition}\label{prop:composition-osc}
    Let $\phi \in \symbcl^1(\RR^d)$ homogeneous of degree $1$ outside a compact set, $a \in \symbcl^{m_1}(\RR^d)$, and $b \in \symbcl^{m_2}(\RR^d)$.
    For any symmetric matrix $A \in \RR^{d \times d}$ we have
    \begin{align*}
        (e^{i\ang{A \cdot, \cdot}} a \# e^{i\phi} b)(z) = e^{i\ang{Az,z}} e^{i\phi(z)} \tilde{a},
    \end{align*}
    where $\tilde a \in \symbcl^{m_1+m_2}(\RR^d)$.
\end{proposition}
\begin{proof}
    Using an approximation argument, we can show that the previous calculation also holds for isotropic symbols and the evaluation of the
    oscillatory integral is essentially the same as in the proof of Lemma 4.2 in \cite{DGW}.
\end{proof}

\subsection{\texorpdfstring{Parametrix for $U(t)$}{Parametrix for U(t)}}
We use Lemma~\ref{lem:average} to simplify the computations. The trace is invariant under conjugation by unitary operators,
therefore we may assume that for $N \gg 0$,
\begin{align*}
    H &= H_0 + B,\\
    [H_0, B] &\in \calccl^{-N}(\RR^d).
\end{align*}
For our purposes it will suffice to have $N = 1$. By Lemma~\ref{lem:average}, the Weyl-quantized symbol $b$ of $B$ is given by
\begin{align*}
    b = \Xray p + \symbcl^{-1}(\RR^d).
\end{align*}
Since the symbol $p$ is assumed to be classical, we have an asymptotic expansion
\begin{align*}
    b \sim \sum_{j=0}^\infty b_{1-j}
\end{align*}
and $b_1 = \Xray \sigma^1(P)$.

In this section, we construct a parametrix for $U(t) = e^{-itH}$, where $H = H_0 + B$.
As before, we first consider the reduced propagator $\rd(t) = U_0(-t) U(t)$, where
$U_0(t) = e^{-itH_0}$ is the propagator of the harmonic oscillator.

The reduced propagator $\rd(t)$ satisfies
\begin{equation}\label{eq:reduced}
    \left\{\begin{aligned}
        (i\pa_t - B(t))\rd(t) &= 0,\\
        \rd(0) &= \id.
    \end{aligned}\right.
\end{equation}
Here, $B(t) = U_0(-t) B U_0(t) = \Opw(b(t))$, where
\begin{align*}
    b(t) = b \circ \exp(t\hamvf_0).
\end{align*}
Note that the first terms in the asymptotic expansion are $b_1$ and $b_0$, repectively, because due to the averaging $b$ is invariant under the flow $\exp(t\hamvf_0)$ modulo $\symbcl^{-N}$.

Our ansatz is
\begin{align*}
    \trd(t) = \Opw(e^{i\phi_1(t)}a(t)),
\end{align*}
where $\phi_1$ is homogeneous of degree $1$ and $a \in \CI(\RR_t, \symbcl^0)$.
Applying $i\pa - B(t)$ to $\trd(t)$ yields a Weyl-quantized operator with full ``symbol''
\begin{align*}
    - e^{i\phi} a(t) \pa_t \phi_1 + ie^{i\phi}\pa_t a - b(t) \# e^{i\phi_1} a.
\end{align*}
Thus, $\trd(t)$ solves \eqref{eq:reduced} if
\begin{align*}
    - e^{i\phi} a(t) \pa_t \phi_1 + ie^{i\phi}\pa_t a - b(t) \# e^{i\phi_1} a = 0.
\end{align*}
Ordering the the terms by homogeneity, we obtain by Proposition \ref{prop:composition-fio} for the leading order:
\begin{align*}
    \pa_t \phi_1 + b_1 = 0,
\end{align*}
which is our usual eikonal equation. Next order gives the first transport equation:
\begin{align*}
    i\pa_t a_0 = (b_0 - (1/2)\{b_1, \phi_1(t)\})a_0.
\end{align*}
The higher transport equations for $a_{-k}$ contain as usual inhomogeneous terms depending on the derivatives of $a_{-j}$, $j < k$.

Hence, the parametrix is given on an interval $(2\pi k - \ep, 2\pi k + \ep)$ by
\begin{align*}
    \trd(t) = \Opw(e^{i\phi_1(t)} a(t)),
\end{align*}
where
\begin{align*}
    \phi_1(t,x,\eta) &= -tb_1(x,\eta),\\
    a_0(t,x,\eta) &= e^{-itb_0(x,\eta)}.
\end{align*}
\begin{proposition}\label{prop:parametrix}
    There is an oscillatory integral operator $\tilde U \in \CI( (2\pi k - \ep,2\pi k+\ep), \mathcal{L}(\schwartz',\schwartz'))$ such that
    $\tilde{U}(t)$ is a parametrix for $U(t)$, that is
    \[\tilde{U}(t) - U(t) \in \CI( (2\pi k - \ep, 2\pi k + \ep), \mathcal{L}(\schwartz',\schwartz))\]
    and
    \begin{align*}
        \tilde{U}(t) = \Opw( e^{i\phi_2(t)} e^{i\phi_1(t)} a(t)),
    \end{align*}
    where $\phi_2(t,x,\eta) = -2\tan(t/2) p_2(x,\eta)$, $\phi_1(t,x,\eta) = -tb_1(x,\eta)$, and
    $a \in \CI( (2\pi k - \ep, 2\pi k + \ep), \symbcl^0(\RR^{2d}))$ with $\sigma^0(a(2\pi k)) = (-1)^{dk} \exp(-2\pi i k b_0)$.
\end{proposition}
\begin{proof}
    It is well-known (cf. Hörmander~\cite{Hormander95}*{p. 427}) that the propagator of the harmonic oscillator can be written as a Weyl-quantized operator. Namely, we have for $t \not \in \pi + 2\pi\ZZ$ that
    \begin{align*}
        U_0(t) = \cos(t/2)^{-d}\Opw(e^{i\phi_2(t)}),
    \end{align*}
    where $\phi_2(t,x,\eta) = -2\tan(t/2) p_2(x,\eta)$.
    A suitable parametrix is therefore given by
    \begin{align*}
        \tilde{U}(t) = U_0(t) \trd(t).
    \end{align*}
    By Proposition~\ref{prop:composition-osc}, $\tilde{U}(t)$ can be represented as a Weyl-quantized operator with symbol
    \begin{align*}
        e^{i(\phi_2(t) + \phi_1(t))} \tilde{a}(t),
    \end{align*}
    where $\tilde{a} \in \CI( (2\pi k - \ep, 2\pi k+\ep), \symbcl^0)$.
    Thus, it remains to show that $\tilde{a}(2\pi k) = (-1)^{dk} a_0(2\pi k) + \symbcl^{-1}$, but
    this obviously true, since $U_0(2\pi k) = (-1)^{dk} \id$.
\end{proof}

\subsection{Proof of Theorem~\ref{thm:principal-invariant}}
We consider an oscillatory integral of the form
\begin{align*}
    I(\lambda) = \int e^{i(t\lambda + \psi_2(t,x,\eta) + \psi_1(t,x,\eta))} \chi(t) a(t,x,\eta) \,dt\,dx\,d\eta.
\end{align*}
We follow the proof of \cite{DGW}*{Proposition 5.1}, but keeping track of the leading order of the amplitude.
We assume that
\begin{itemize}
    \item $\chi \in \CcI(\RR)$,
    \item $\psi_j$ is homogeneous of degree $j$ outside a compact neighborhood of $0$,
    \item $a \in \CI(\RR, \symbcl^0(\RR^d))$, and $\psi_j$ are smooth on the support of $a$,
    \item there exists unique $t_0 \in \supp \chi$ and $r_0 > 0$ such that
        \begin{align*}
            \psi_2(t_0, r_0, \theta) &= 0,\\
            \pa_t \psi_2(t_0,r_0,\theta) &= -1
        \end{align*}
        for all $\theta \in \Sph^{2d-1}$,
    \item $\psi_2$ is normalized in the sense that $|\pa_r\pa_t \psi_2(t_0,r_0,\theta)| = r_0$.
\end{itemize}
The extension of \cite{DGW}*{Proposition 5.1} is the following:
\begin{proposition}\label{prop:stat-phase}
    Under the assumptions above and assuming that $\psi_1(t_0,r_0,\bullet)$ is Morse-Bott with $2d-2$ non-degenerate directions,
    the integral $I(\lambda)$ has an asymptotic expansion
    \begin{align*}
        I(\lambda) = \lambda^{(d-1)/2}e^{it_0\lambda} \sum_{j=1}^n e^{i\lambda^{1/2}\psi_1(t_0,r_0,\theta_j)} \sum_{l=0}^\infty \lambda^{-l/2}\gamma_{j,l},
    \end{align*}
    where
    \begin{align*}
        \gamma_{j,0} = \frac{(2\pi)^d}{\abs{\det D_\theta d_\theta\psi_1(t_0,r_0,\theta_j)}^{1/2}} r_0^{2(d-1)} e^{\pi i \sigma_j/4} \int \sigma^0(a(t_0))(r_0,\theta_j)d\theta.
    \end{align*}
    Here, $\sigma_j$ is the signature of $D_\theta d_\theta \psi_1(t_0,r_0,\theta_j)$ and the integral is over
    the $1$-dimensional manifold on which $\psi_1$ is constant.
\end{proposition}
\begin{proof}
    It already follows from the proof of \cite{DGW}*{Proposition 5.1} that we have the claimed asymptotic expansion and we only have to calculate the leading coefficient.
    As in \cite{DGW}, we have that
    \begin{align*}
        I(\lambda) = \int_{\Sph^{2d-1}} J(\lambda, \lambda^{-1/2}, \theta) d\theta,
    \end{align*}
    where
    \begin{align*}
        J(\lambda, \mu, \theta) &= \lambda^d \int e^{i\lambda \Psi_\mu(t,r,\theta)} \chi(t) a(t,\lambda^{1/2}r,\theta) r^{2d-1} dt\,dr,\\
        \Psi_\mu(t,r,\theta) &= \psi_2(t,r,\theta) + \mu \psi_1(t,r,\mu) + t.
    \end{align*}

    By assumption, the determinant of $D_{r,t}d_{r,t} \psi_2(t_0,r_0,\theta)$ has absolute value $r_0$ and the signature is zero,
    hence, by the stationary phase formula, we obtain for any $M \geq 1$,
    \begin{align*}
        J(\lambda,\mu,\theta) = \lambda^{d-1} e^{i\lambda (t_0 + \mu \psi_1(t_0,r_0,\theta))} a_M(\lambda^{1/2},\mu,\theta) + O(\lambda^{d-1-M})
    \end{align*}
    with
    \begin{align*}
        a_M(\lambda^{1/2}, \lambda^{-1/2}, \theta) = 2\pi r_0^{2(d-1)} \sigma^0(a(t_0))(r_0,\theta) + O(\lambda^{-1/2}).
    \end{align*}
    This leads to the asymptotic formula
    \begin{align*}
        I(\lambda) = \lambda^{d-1} e^{i\lambda t_0} \cdot 2\pi r_0^{2(d-1)} \int e^{i\lambda^{1/2} \psi_1(t_0,r_0,\theta)} (\sigma^0(a(t_0))(r_0,\theta) + O(\lambda^{-1/2}) ) d\theta.
    \end{align*}
    By assumption, $\psi_1(t_0,r_0,\theta)$ is Morse-Bott, so we may apply the stationary phase formula again and
    obtain
    \begin{align*}
        I(\lambda) = \lambda^{(d-1)/2} e^{i\lambda t_0} \sum_{j=1}^n e^{i\lambda^{1/2} \psi_1(t_0,r_0,\theta_j)} ( \gamma_{j,0} + O(\lambda^{-1/2}) ).
    \end{align*}
\end{proof}

Now, we are able to prove the main theorem.
\begin{proof}[Proof of Theorem \ref{thm:principal-invariant}]
    Let $\ep > 0$ small enough 
    and choose $\chi \in \CcI( (-\ep+2\pi k,\ep+2\pi k))$ such that $\chi(2\pi k) = 1$ and $k \in \ZZ\setminus \{0\}$.
    By \cite{CoRo}*{Proposition 13}, we can calculate the trace of a Weyl-quantized operator by integration along the diagonal.
    Thus, the inverse Fourier transform of the Schrödinger trace is given by
    \begin{align*}
        \F^{-1}_{t \to \lambda} \chi(t)\Tr U(t) &= \F^{-1}_{t \to \lambda} \Tr \tilde{U}(t) + O(\lambda^{-\infty})\\
        &= (2\pi)^{-(d+1)} \int e^{i (\phi_2(t,x,\xi) + \phi_1(t,x,\xi) + t\lambda) } \chi(t)a(t,x,\xi) dt\,dx\,d\xi + O(\lambda^{-\infty}),
    \end{align*}
    where $\phi_2, \phi_1$, and $a$ are given by Proposition \ref{prop:parametrix}.
    The phase function $\phi_2$ has an expansion $\phi(t - 2\pi k,r,\theta) = -(1/2)r^2t + O(t^3)$ and therefore we may apply Proposition~\ref{prop:stat-phase} with
    \begin{align*}
        I_k(\lambda) = (2\pi)^{-(d+1)}\int e^{i (\phi_2(t,x,\xi) + \phi_1(t,x,\xi) + t\lambda) } \chi(t)a(t,x,\xi) dt\,dx\,d\xi.
    \end{align*}
    The stationary point of $\phi_2$ is at $t_0 = 2\pi k$ and $r_0 = \sqrt{2}$.
    On the stationary point, we have that $\phi_1(t_0,r_0,\cdot) = -2\pi k b_1(\sqrt{2}, \cdot)$.
    By Proposition~\ref{prop:stat-phase}, we see that there is an asymptotic expansion
    \begin{align*}
        I_k(\lambda) = \lambda^{(d-1)/2}e^{2\pi ik\lambda} \sum_{j=1}^n e^{-2\pi ik\lambda^{1/2}b_1(\sqrt{2},z_j)} \sum_{l=0}^\infty \lambda^{-l/2}\gamma_{k,j,l}.
    \end{align*}
    The determinant and signature of the Hessian of the phase function $\phi_1$ at a critical point $z_j$ are given by $(2\pi k)^{2d-2} d_j$ and $-\sigma_j$, respectively, where
    \begin{align*}
        d_j &= \abs{\det(Dd b_1(z_j))} \neq 0,\\
        \sigma_j &= \sgn(Dd b_1(z_j)).
    \end{align*}
    Since $\sigma^0(a(2\pi k)) = (-1)^{dk} e^{-2\pi i k b_0}$ is invariant under the flow, we have for the leading order term at one critical point $z_j$ that
    \begin{align*}
        \gamma_{k,j,0} &= (2\pi)^{-(d+1)}\frac{(2\pi)^d}{(2\pi k)^{d-1}\abs{d_j}^{1/2}} 2^{d-1} e^{-\pi i \sigma_j/4} \int \sigma^0(a(2\pi k))(\sqrt{2},z_j)d\theta\\
        &= (\pi k)^{-(d-1)} d_j^{-1/2} e^{-\pi i \sigma_j/4} \cdot \frac{1}{2\pi} \int \sigma^0(a(2\pi k))(\sqrt{2},z_j)d\theta\\
        &= (\pi k)^{-(d-1)} d_j^{-1/2} e^{\pi i (-\sigma_j/4 + dk)} e^{-2\pi i k b_0(z_j)}.
    \end{align*}
\end{proof}

\section{Conjugation to the Bargmann-Fock Model} 
In this section, we prove Theorem \ref{INTROTHEO} as outlined in Section
\ref{BFINTRO}. The proof is based on results of \cites{ZeZh17,ZeZh18-1} and involves Toeplitz Fourier integral operators acting on holomorphic sections
of the standard line bundles $\ocal(N) \to \CP^{d-1}$ over the projective space. 
Since it is our second proof of the main result, we do not provide detailed background on  Toeplitz Fourier integral operators and refer to  \cites{ZeZh17,ZeZh18-1}
for further background and references. Our goal to describe the conjugation to the holomorphic setting and to connect the asymptotics of Theorem \ref{INTROTHEO} to those of \cites{ZeZh17,ZeZh18-1}.

In this section, we assume that the perturbation $B$ commutes with $H_0$.
As in Lemma~\ref{lem:average-smoothing}, by iterated averaging we may assume
that $H = H_0 + B + R$ where $R$ is smoothing and $[H_0, B] = 0$. By Corollary~\ref{cor:smoothing} and Proposition \ref{PROPA1}, the
smoothing operator does not change singularities or asymptotics of the trace;
it is omitted for simplicity of exposition.

As in \eqref{PiN}, we denote by $\hcal_N$ the eigenspace of $H_0$ with the eigenvalue $N+d/2$ and by $\Pi_N$ the orthogonal projection onto $\hcal_N$. The perturbation $P$ and the unitarily equivalent $B$  are  isotropic pseudo-differential operators of order $1$. When we conjugate to
the Bargmann-Fock setting, orders in Toeplitz calculus  are traditionally defined by powers of $N$. Thus, $H_0$ is considered to have order $1$ and $P, B$ are considered to have order $\half$. 
To keep track of the orders, it is convenient to scale $B$ to have order $0$,
i.e. 
we define zeroth order isotropic pseudo-differential operator  \[\tilde{B} \coloneqq H_0^{-1/2} B\] with Weyl-quantized symbol $\tilde{b} = p_2^{-1/2} \# b.$
Then $[\tilde{B}, H_0] = 0$, and \[\Tr e^{i t B}|_{\hcal_N}  = \Tr e^{i t H_0^{1/2} \tilde{B}}|_{\hcal_N} = \Tr e^{i t \sqrt{N} \tilde{B}}|_{\hcal_N}.\]
As mentioned in the introduction, $e^{i t \sqrt{N} \tilde{B}}|_{\hcal_N}$ is not
a standard type of Fourier integral operator, which would be the exponential
of a first order operator, and that is why we view it as $U(\frac{t}{\sqrt{N}})$
(in the notation of Section \ref{BFINTRO}).

\subsection{Conjugation to Bargmann-Fock Space}
Consider the weight function $\weight(z) = |z|^2/2$. The $L^2$-space of weighted entire functions,
\[H_\weight(\CC^d) \coloneqq L^2(\CC^d, e^{-2\weight(z)/h} d^{2d}z) \cap \operatorname{Hol}(\CC^d)\] is called the \emph{Bargmann-Fock space}. Here, $d^{2d}z$ denotes the Lebesgue measure on $\CC^d$.

There exists a standard unitary intertwining operator, the \emph{Bargmann transform}\footnote{It is a special case of  the FBI transform.},  from $L^2(\R^d)$ to the Bargmann-Fock space, 
\[\barg: L^2(\R^d) \to H_\weight(\CC^d),\]
defined by
\[\barg u (z; h) =  2^{-d/4}(\pi h)^{-3d/4}\int_{\R^d} e^{i\varphi(z,y)/h}  u(y) dy\]
with phase function
\[\phase(z, x) = i \left(\frac{1}{2} (z^2 + x^2) - \sqrt{2} x \cdot z\right).\]
It is a Fourier integral operator with positive complex phase.
The phase function and the weight are related by
\[\weight(z) = \sup_{y \in \RR^d} (-\Im \phase(z,x)).\]
We denote by $\Pi^{BF}$ the orthogonal projection $ L^2(\CC^d, e^{-2\weight}d^{2d}z) \to H_\weight(\CC^d)$.
Its Schwartz kernel is known as the Bargmann-Fock Bergman kernel.
\begin{remark}
    The operator $\barg$ can be written as $\barg f(z;h) = (\mathscr{B}_{1/(2h)} f)(\sqrt{2} z)$, where
    $\mathscr{B}_\alpha$ is defined as in Zhu~\cite{Zhu}*{Sect. 6.2}.
\end{remark}
\begin{remark}
    In what follows, we will always take $h = 1$.
\end{remark}

Define the linear complex canonical transformation
\begin{align*}
    \trafo : \CC^{2d} &\to \CC^{2d}\\
    (x,\xi) &\mapsto \frac{1}{\sqrt{2}} (x-i\xi, \xi - ix),
\end{align*}
which satisfies $\trafo(x,-\pa_x \phase(z,x)) = (z,\pa_z\phase(z,x))$ for all $x,z\in \CC^d$.
The canonical transformation maps $\RR^{2d}$ bijectively onto the totally real IR Lagrangian subspace
\begin{align*}
    \Lagr &\coloneqq \left\{ (z,-2i \pa_z \weight(z)) \colon z \in \CC^d\right\} = \{(z,-i\bar{z}) \colon z \in \CC^d\} \subset \CC^{2d}.
\end{align*}
For a proof, we refer to \cite{Zworski12}*{Theorem 13.5}.
The inverse of $\trafo : \RR^{2d} \to \Lagr$ is given by
\begin{align*}
    \trafo^{-1} : \Lagr &\to \RR^{2d}\\
    (z,\zeta) &\mapsto \frac{1}{\sqrt{2}} (z+i\zeta, \zeta + iz).
\end{align*}
The Bargmann transform is a quantization of $\trafo$ in the sense that
\begin{align*}
    \barg^* a^w(z,D_z) \barg = (\trafo^* a)^w(x,D_x), \quad a \in \symb(\Lagr)
\end{align*}
(cf. Zworski~\cite{Zworski12}*{Theorem 13.9} in the semiclassical setting).

It is a classical fact (cf. Zhu~\cite{Zhu} or Zworski~\cite{Zworski12}*{Theorem 4.5}) that   $H_0 - d/2$ is conjugated to the degree operator $\ncal \coloneqq \ang{z, \pa_z}$ on Bargmann-Fock space.
Note that $(\trafo^* p_2)(z,\zeta) = \ang{z,i\zeta}$, where $p_2(x,\xi) = (1/2)(|x|^2 + |\xi|^2)$.
This gives a short proof that $\barg^* (H_0 - d/2) \barg = \ncal$, since $\Opw(\ang{z,i\zeta}) = \ncal + d/2$.

The eigenspace $\hcal_N^{BF}$  of eigenvalue $N$ is spanned by the monomials $z^{\alpha}$ with $|\alpha| = N$.
Comparing the spectral decompositions of $H_0$ and the degree operator gives
\begin{lemma} \label{CONJLEM2}
    The operator $\Pi_N^{BF} \coloneqq \barg \Pi_N \barg^*$
    is the orthogonal projection onto $\hcal_N^{BF}$.
\end{lemma}

Under conjugation by $\barg$, $H_0 \tilde{B}$ transforms as $\barg H_0 \tilde{B} \Pi_N \barg^* = \ncal \barg \tilde{B} \barg^*$.
It follows that
\begin{equation} \label{CONJ}
    \barg \Pi_N H_0 \tilde{B} \Pi_N \barg^* = (N + \frac{d}{2}) \Pi_N^{BF} \barg \tilde{B} \barg^* \Pi_N^{BF}.
\end{equation}
\subsection{Toeplitz Operators}
The next step is to determine the operator $\Pi_N^{BF}  \barg \tilde{B} \barg^* \Pi_N^{BF} $.
We now review the Bargmann conjugation of Weyl pseudo-differential operators to Toeplitz
operators, following \cites{Guillemin84,CHSj18,Zworski12}.
Our presentation differs from these references in two ways:
(i) we are dealing with isotropic pseudo-differential operators, which are considered in \cite{Guillemin84} but not in the other two references;
(ii) we are interested in semi-classical asymptotics in $N$ rather than in homogeneous operators.
Hence, we need to reformulate results of the references in terms of semi-classical symbol expansions.
The basic relation is that if $a_j(z, \bar{z})$ is homogeneous
in $(z, \bar{z})$ of order $j$, then the semi-classical Toeplitz operator 
$\Pi_N a_j \Pi_N$ is of order $N^{j/2}$ (cf. Lemma~\ref{UCONJ} below).

We also recall that there exist two notions of complete symbol for a Toeplitz operator:
(i) its contravariant symbol $q$, and
(ii) its covariant symbol $\Pi^{BF} q \Pi^{BF}(z,z)$, i.e. the value of the Schwartz kernel on the (anti-)diagonal.
The transform from the contravariant symbol to the covariant symbol is known as the Berezin transform.

The Bargmann transform conjugates isotropic Weyl pseudo-differential operators
$\Opw(a)$ to Toeplitz operators $\Pi^{BF} q \Pi^{BF}$ with isotropic symbols.
The complete contravariant symbol of an isotropic Toeplitz operator $\Pi^{BF} q \Pi^{BF}$ of order $0$ 
is an isotropic symbol on $\C^{d} \simeq T^*\R^d$ of the form,
\[q \sim q_0 + q_{-1} + q_{-2} + \cdots,\]
where $q_{-j}$ is homogeneous of degree $-j$.

By a well-known result of Berezin (see \cite{CHSj18} for a recent proof), the relation between the  complete Weyl symbol $a$ of $\Opw(a)$ to the contravariant symbol $q$
is given by
\begin{equation}\label{t}
    a(x, -2i \pa_x \weight(x)) = \left( \exp (\frac{1}{4} (\pa_{x\bar{x}}\weight)^{-1} \partial_x \cdot \partial_{\bar{x}}) q \right)(x), \qquad x\in \C^d.
\end{equation}
The operator $(\pa_{x\bar{x}}\weight)^{-1} \partial_x \cdot \partial_{\bar{x}}$  is a constant coefficient
second order differential operator on $\C^n$ whose symbol is a negative definite quadratic form; so this  is a forward heat flow acting on $q$.
The Berezin transform is the inverse heat flow.
In our case (cf. \cites{CHSj18,Guillemin84,Zworski12}), we have that $(1/4)(\pa_{x\bar{x}} \weight)^{-1} \pa_x \cdot \pa_{\bar{x}} = (1/8) \Delta$ (the standard Euclidean Laplacian on $\R^{2d}$) and therefore
\begin{equation} \label{p}
    q = e^{-\frac{1}{8} \Delta} a \circ \trafo.
\end{equation}
A general $\CI$ isotropic symbol of order $0$ does not lie in the 
domain of $ e^{-\frac{1}{8} \Delta} $. However,  the expression \eqref{p} makes
sense as an isotropic symbol since $\Delta$ lowers the order of an isotropic symbol by two orders. That is, we invert the transform \eqref{t} in the topology of symbols and
view $ e^{-\frac{1}{8} \Delta} $ as an operator taking complete (formal) isotropic Weyl symbols to complete (formal) isotropic contravariant symbols.
If we Taylor expand $e^{-\frac{1}{8} \Delta}$ to order $M$, and $a$ is isotropic of order $m$, then
\[ e^{-\frac{1}{8} \Delta} a = \sum_{k=0}^M \frac{(-1)^k}{8^k k!} (\Delta^k a),\quad \text{ mod }\; \symb^{m-2M}.\]

Summing up, we have the following:
\begin{lemma} \label{CONJLEM}
    Let $A = \Opw(a)$ be a zeroth order isotropic pseudo-differential operator on $\R^d$.
    Then   $\barg A \barg^*$ is a Toeplitz operator $\Pi^{BF} q \Pi^{BF} $ on $H_\weight(\C^d)$,
    where the symbol $q$ has an asymptotic expansion
    \begin{align*}
        q &\sim \sum_{k=0}^\infty \frac{(-1)^k}{8^k k!} (\Delta^k a) \circ \trafo
    \end{align*}
    in the sense of isotropic symbols.

    If we assume that $a \in \symbcl^0$ and $a \sim \sum_k a_{-k}$, then $q \in \symbcl^0$ with expansion $q \sim \sum_k q_{-k}$, where
    \begin{align*}
        q_0 &= a_0 \circ \trafo\\
        q_{-1} &= a_{-1} \circ \trafo\\
        q_{-2} &= (a_{-2} - (1/8)\Delta a_0) \circ \trafo.
    \end{align*}
\end{lemma}

In particular, we have that \[\ncal = \Pi^{BF} (|z|^2-d) \Pi^{BF},\]
since the symbol of $\ncal$ restricted to $\Lagr$ is $a(z)=|z|^2 - d/2$ and the Berenzin transform of $a$ is $q(z) = |z|^2 - d$.

Returning to the symbol $p \in \symbcl^0$ with $\{p, p_2\} = 0$ and asymptotic expansion
\begin{align*}
    p \sim \sum_{j=0}^\infty p_{-j}.
\end{align*}
We now compress $\Pi^{BF} p \Pi^{BF} $ with $\Pi_N^{BF}$ and exponentiate.
The order of a homogeneous isotropic symbol coincides
with its eigenvalue under the degree operator $\ncal$.
This is an immediate consequence of Euler's homogeneity theorem:
If we write $z = (x + i\xi)$, then
\begin{align*}
    \ncal p &= \frac{1}{2}(x\pa_x + \xi\pa_\xi) p +  \frac{i}{2}(\xi\pa_x - x\pa_\xi)p\\
    &= \frac{1}{2}(x\pa_x + \xi\pa_\xi)p + \frac{i}{2}\{p_2,p\}.
\end{align*}
Since $p$ was assumed to Poisson-commute with $p_2$, the second term vanishes and we obtain that
$\ncal p_j = (j/2) p_j$.
We may write
\begin{align*}
    \Pi^{BF}_N p_j \Pi^{BF}_N &= \Pi^{BF}_N (\ncal + d)^{j/2} |z|^{-j} p_j \Pi^{BF}_N\\
    &= (N+d)^{j/2} \Pi_N^{BF} |z|^{-j} p_j \Pi^{BF}_N.
\end{align*}
Note that $|z|^{-j} p_j$ is bounded with norm independent of $N$.
This agrees with the statement at the beginning of this section that isotropic orders
get multiplied by $\half$ when we conjugate to the Bargmann-Fock model and use homogeneity in $H_0 \simeq \ncal$ to define orders. It follows that the polyhomogeneous expansion
of an isotropic symbol coincides with its expansion in powers of $(N + d)^{-1/2}$ when compressed by $\Pi_N^{BF}$. We thus have:

\begin{lemma} \label{UCONJ}
    The operator $\Pi_N^{BF} \barg \tilde{B} \barg^{-1} \Pi_N^{BF}$ is a
    semi-classical Toeplitz operator whose complete contravariant expansion
    has the form
    \[q = p_0 \circ \trafo + N^{-1/2}p_{-1} \circ \trafo + N^{-1} (p_{-2} - (1/8) \Delta p_0) \circ \trafo + O(N^{-3/2}).\]

    In particular, we have that $\Pi_N \tilde{B} \Pi_N$ is conjugated to $\Pi^{BF}_N q \Pi^{BF}_N$ under the Bargmann transform.
\end{lemma}

\subsection{Trace Asymptotics}

Our aim is to determine the large $N$ asymptotics of 
\begin{equation} \label{BTRACE2}
    \Tr  \Pi_N e^{i t \sqrt{N} \tilde{B}}\Pi_N.
\end{equation}
We observe that \eqref{BTRACE2} is the trace of the rescaled propagator $U_N(\frac{t}{\sqrt{N}})$, where
\begin{equation}\label{UNDEF}
    U_N(t) \coloneqq \Pi_N  e^{it N \tilde{B}}\Pi_N.
\end{equation}
It follows from Lemma~\ref{UCONJ} that $U_N(t) = \barg^* \exp( it N \Pi^{BF}_N q \Pi^{BF}_N) \barg$ is a semi-classical isotropic Fourier integral operator.
Clearly,
\begin{equation} \label{SCALED}
    \Tr  \Pi_N e^{i t B} = \Tr U_N(\frac{t}{\sqrt{N}}).
\end{equation}
     
The main value of the Bargmann-Fock conjugation is that the propagator of the harmonic
oscillator and the projection $\Pi_N$ become much simpler on the Bargmann-Fock side.
The trace $\Tr U_N(\frac{t}{\sqrt{N}})$ may be further simplified by noting that $\hcal_N$
is the same as the  space $H^0(\CP^{d-1}, \ocal(N))$ of holomorphic sections
of the $N$th power of the natural line bundle $\ocal(1) \to \CP^{d-1}$.
The identification is to lift holomorphic sections, $s \to \hat{s}$,  of $\ocal(N)$ to homogeneous 
functions on $\C^d\setminus\{0\}$. It follows that
\begin{equation} \label{BFDECOM}
    H^2(\C^d, e^{-|z|^2} d^{2d}z ) \simeq   \bigoplus_{N=0}^{\infty} \hcal^{BF}_N \simeq \bigoplus_{N=0}^{\infty} H^0(\CP^{d-1}, \ocal(N)).
\end{equation}
We refer to \cite{GH} for background.
Tracing through the identifications, we see that $\hcal_N \simeq H^0(\CP^{d-1}, \ocal(N))$.
This identification explains why the trace formula is an integral over the space of Hamilton orbits of $H_0$. 

We use the last identification to determine the asymtotics of the trace of $U_N(\frac{t}{\sqrt{N}})$  in the model $H^0(\CP^{d-1}, \ocal(N))$. 
The advantage of conjugating to this model is that the calculations have
mostly been done in this setting in \cites{ZeZh17, ZeZh18-1} (in fact, on any \kahler manifold $M$).
It would be equivalent to work directly with the Fourier components $\Pi_N^{BF} \tilde{B} \Pi_N^{BF}$.
 
We let $\Pi_N^{\CP^{d-1}}: L^2(\CP^{d-1}, \ocal(N)) \to H^0(\CP^{d-1}, \ocal(N))$ denote the orthogonal projection. Since $\tilde{b}$ is invariant
under the natural $S^1$ action on $\C^d$ defining $\C^d \setminus \{0\} \to \CP^{d-1}$, $\sigma^0(\tilde{b})$ descends to a multiplication operator on $L^2(\CP^{d-1}, \ocal(N))$.
 
We briefly recall the setting of \cites{ZeZh17,ZeZh18-1}.
Consider a polarized \kahler manifold $M$ with positive Hermitian line bundle $L$ (cf. \cite{ZeZh18-1} for definitions) and let
$H^0(M, L^N)$ denote the space of holomorphic sections of the $N$-th power of the line bundle $L$ and $\Pi^M_N$ is the orthogonal projection $L^2(M,L^N) \to H^0(M,L^N)$.
For any function $H : M \to \RR$ we define
\[\h H_N = \Pi^M_N \, H \, \Pi^M_N : H^0(M, L^N) \to H^0(M, L^N)\]
the corresponding semiclassical Toeplitz operator and
\begin{equation} \label{Ukt}
    U_N(t) = \exp i t N \h H_N
\end{equation}
its propagator, which is a semiclassical Fourier integral operator.
In \cites{ZeZh17, ZeZh18-1} it is shown that for any $z \in M$, the following pointwise asymptotics hold:
\begin{proposition}[\cite{ZeZh17}*{Proposition 5.3}] \label{ZeZh17}
    Let $(M, \omega)$ be a \kahler manifold of complex dimension $m$ and $H : M \to \RR$ a Morse function.
    If $z \in M$, then for any $\tau \in \R$,
    \[ U_N(t/\sqrt{N}, z , z ) = \kk^{m} e^{i t \sqrt{N} H(z )} e^{-t^2 \frac{\| dH(z )\|^2}{4}} (1 + O(|t|^3N^{-1/2})), \]
    where the constant in the error term is uniform as $t$ varies over compact subset of $\R$. 
\end{proposition}
\begin{remark}
    It is emphasized that the asymptotics are valid at critical points of $H$.
\end{remark}

The asymptotics of the trace follow from Proposition \ref{ZeZh17} and
the method of stationary phase.
\begin{theorem}[cf. \cite{ZeZh18-1}*{Theorem 1.7}]\label{TR}
    Let $(M, \omega)$ be a \kahler manifold of complex dimension $m$.
    If $t \neq 0$, the trace of the scaled propagator $U_N(t/ \sqrt{N}) = e^{i \sqrt{N} t \h H_N}$ admits the following aymptotic expansion
    \begin{align*}
        \int_{z \in M} U_N(t/\sqrt{N}, z,z) d \Vol_M(z)  &= N^m  \left(\frac{t \sqrt{N}}{4 \pi}\right)^{-m} \sum_{z_c \in \crit(H)} \frac{e^{i t \sqrt{N} H(z_c)}e^{(i\pi/4) \sgn(DdH(z_c))}}{\sqrt{|\det(DdH(z_c)))|}}\\
        &\phantom{=} \cdot (1 + O(|t|^3N^{-1/2})).
    \end{align*}
\end{theorem}

We note that the Gaussian factor equals $1$ at the critical points.
In particular, the theorem applies to $M = \CP^{d-1}$, $L = \ocal(1)$, and $H = \tilde\Xray p_1$.
A complete asymptotic expansion with remainder could be obtained by
the same method if one wished to have lower order terms. This proves 
Theorem \ref{INTROTHEO}.

\begin{corollary} \label{TR2}
    Let $\tilde{b}_0$ be the principal isotropic symbol of 
    $\tilde{B} = H^{-\half}_0 B$, viewed as a function on $\CP^{d-1}$.
    Then if $\tilde{b}_0$ is a Morse function on 
    $\CP^{d-1}$, the asymptotics of \eqref{SCALED} are given by Theorem~\ref{TR} with $m = d-1$.
\end{corollary}

\section{Equivalence of the Expansions}\label{EQUIVSECT}
In this section we show that the large $\lambda$ expansion of
Theorem \ref{thm:principal-invariant} agrees with the large $N$ expansion
of Theorem \ref{INTROTHEO} to leading order. 

Consider the Fourier series
\[w_r(t, a) \coloneqq \sum_{N=1}^{\infty} N^{-r}  e^{-i (N + d/2) t} e^{- i a t \sqrt{N}}. \]
We introduce a cutoff $\hat{\rho}(t)$ supported near $t = 2\pi k_0$ with $\hat\rho(2\pi k_0) = 1$ and calculate asymptotics as $\lambda \to \infty$ of 
\begin{equation} \label{INT}
    \begin{aligned}
        \F^{-1}_{t \to \lambda}\{\hat\rho(t)w_r(t,a)\}(\lambda) &= (2\pi)^{-1}\int_{\R} \hat{\rho}(t) e^{i t \lambda} w_r(t,a) dt\\
        &= (2\pi)^{-1}\sum_{N=1}^{\infty} N^{-r} \int_{\R} \hat{\rho}(t) e^{i t \lambda}   e^{-i t (N + d/2) } e^{- i a t\sqrt{N}} dt.
    \end{aligned}
\end{equation}
We claim that the large $\lambda$ asymptotics of the integral \eqref{INT} coincides with the large $N$ asymptotics of the Fourier coefficients. 
\begin{proposition}
    Let $w_r$ and $\rho$ be as above. Then
    \begin{align*}
        \F^{-1}_{t \to \lambda}\{\hat\rho(t) w_r(t,a)\}(\lambda) = \lambda^{-r} e^{i\pi k_0 d} e^{2\pi i k_0 a \lambda^{1/2}} + O(\lambda^{-r-1/2}).
    \end{align*}
\end{proposition}

\begin{proof}
Let $f_r \in \CI(\RR)$ such that $f_r(\xi) = \xi^{-r}$ for $\xi > 1/2$ and $f_r(\xi) = 0$ for $\xi \leq 0$.
By applying the Poisson summation formula to the function \[\xi \mapsto f_{r}(\xi) \F^{-1}_{t \to \lambda}\{ e^{it(\xi + d/2)} e^{iat|\xi|^{1/2}}\}(\lambda),\]
we have that
\begin{align*}
    I(\lambda) \coloneqq&{} \F^{-1}_{t \to \lambda}\{\hat\rho(t) w_r(t,a)\}(\lambda)\\
    =&{} (2\pi)^{-1}\sum_{k \in \ZZ} \int_{\RR} \int_0^{\infty} \hat\rho(t) e^{it\lambda} e^{2\pi i k \xi} f_r(\xi) e^{-iat|\xi|^{1/2}} e^{-it(\xi + d/2)} d\xi\,dt.
\end{align*}
Changing variables $\xi \mapsto \lambda \xi$ for $\lambda > 0$ yields
\begin{align*}
    I(\lambda) = (2\pi)^{-1}\lambda^{1-r} \sum_{k \in \ZZ} \int_{\RR} \int_0^{\infty} \hat\rho(t) e^{i\lambda(-t\xi + t + 2\pi k \xi)} f_r(\xi) e^{-iat\lambda^{1/2}|\xi|^{1/2}} e^{-itd/2} d\xi\,dt + O(\lambda^{-\infty}).
\end{align*}
We set $a_\lambda(t,\xi) = (2\pi)^{-1}\hat\rho(t) f_r(\xi) e^{-iat\lambda^{1/2}|\xi|^{1/2}} e^{-itd/2}$ and $\phi_k(t,\xi) = (1-\xi)t + 2\pi k\xi$
and note that derivatives in $\xi$ decrease the order in $\xi$ by $1/2$ while increasing it in $\lambda$ by $1/2$. This suffices, since integration by parts decreases the orders by $1$.
We have
\begin{align*}
    I(\lambda) = \lambda^{1-r} \sum_{k \in \ZZ} \int_{\RR} \int_0^{\infty} e^{i\lambda\phi_k(t,\xi)} a_\lambda(t,\xi) d\xi\,dt + O(\lambda^{-\infty}).
\end{align*}
We may decompose the integral as follows:
\begin{align*}
    \F_{t \to \lambda}\{\hat\rho(t) w_r^{2\pi k_0}(t,a)\}(\lambda) &= \lambda^{1-r} (I_1(\lambda) + I_2(\lambda) + I_3(\lambda) + O(\lambda^{-\infty})),
\end{align*}
where
\begin{align*}
    I_1(\lambda) &= \int_{\R} \int_0^{\infty}\psi(\xi) e^{i\lambda\phi_{k_0}(t,\xi)} a_\lambda(t,\xi) d\xi  dt,\\
    I_2(\lambda) &= \int_{\R} \int_0^{\infty}(1-\psi(\xi)) e^{i\lambda\phi_{k_0}(t,\xi)} a_\lambda(t, \xi) d\xi  dt,\\
    I_3(\lambda) &= \sum_{k \not = k_0} \int_{\R} \int_0^{\infty}\psi(\xi) e^{i \lambda\phi_k(t,\xi)} a_\lambda(t,\xi) d\xi  dt,\\
    I_4(\lambda) &= \sum_{k \not = k_0} \int_{\R} \int_0^{\infty}(1-\psi(\xi)) e^{i \lambda \phi_k(t,\xi)} a_\lambda(t,\xi) d\xi  dt.
\end{align*}

First, we consider the integral $I_2(\lambda)$. We have that $\pa_t \phi_k(t,\xi) = 1-\xi \not = 0$ on $\supp \psi$. Thus, by integration by parts we obtain arbitrary decay in both $\lambda$ and $\xi$, which yields
that the integral converges and is $O(\lambda^{-\infty})$.

The $I_3(\lambda)$ integral posses no problem, since by integration by parts
For $I_3(\lambda)$, we use that $\pa_\xi \phi_k = 2\pi k-t \not = 0$ for $k \not = k_0$ on the support of $\hat\rho$. Hence, we obtain arbtrary decay in $k$ to make the series converge, this gives also rapid decay in $\lambda$, showing that we have indeed
$I_3(\lambda) = O(\lambda^{-\infty})$.

The last integral, $I_4(\lambda)$, can be treated by a combination of the previous arguments to make the integral and series converge with arbitrary decay in $\lambda$.

Hence, we have that
\begin{align*}
    \F_{t \to \lambda}\{\hat\rho(t) w_r^{2\pi k_0}(t,a)\}(\lambda) &= \lambda^{1-r} I_1(\lambda) + O(\lambda^{-\infty})\\
    &= \lambda^{1-r} \int e^{i\lambda \phi_{k_0}(t,\xi)} \psi(\xi) a_\lambda(t,\xi) d\xi\,dt + O(\lambda^{-\infty}).
\end{align*}
The phase function $\phi_{k_0}$ is stationary on $t = 2\pi k_0$ and $\xi = 1$ and its Hessian is given by
\begin{align*}
    Dd\phi_{k_0}|_{d\phi_{k_0} = 0} = \begin{pmatrix} 0 & 1\\1 & 0\end{pmatrix}.
\end{align*}
The method of stationary phase yields
\begin{align*}
    I_1(\lambda) = \lambda^{-1} e^{2\pi i k_0} e^{-2\pi i k_0 a \lambda^{1/2}} e^{i\pi k_0 d} + O(\lambda^{-3/2}).
\end{align*}

\end{proof}

\appendix
\section{Duhamel's Formula}

We recall the Duhamel principle for general isotropic evolution equations. 
Let $H \in \calc^2(\RR^d)$ elliptic and $R \in \calc^{-\infty}(\RR^d)$ such that $H$ and $H + R$ are self-adjoint.
Consider the propagators $U(t) = e^{-itH}$ and $V(t) = e^{-it(H + R)}$.
\begin{proposition}\label{PROPA1}
    The difference $U(t) - V(t)$ is a smoothing operator.
\end{proposition}
\begin{proof}
    The difference of the propagators, $F(t) = U(t) - V(t)$ solves the equation
    \begin{equation}
        \left\{\begin{aligned}
            i\pa_t F(t) &= (H + R) F(t) + R(t)\\
            F(0) &= 0,
        \end{aligned}\right.
    \end{equation}
    where $R(t) = R U(t)$.
    By the Duhamel principle,
    \[F(t) = \int_0^t V(t-s) R(s) ds.\]
    Since $V(t)$ and $U(t)$ are unitary, we obtain that for any $N \in \NN$,
    \begin{align*}
        \|H_0^N F(t)\|_{L^2} &= \|\int_0^t H_0^N V(t-s) R(s) ds\|_{L^2}\\
        &\leq \int_0^t \|R H_0^N\|_{L^2} ds\\
        &\leq \sup_{s \in [0,t]} \|R H_0^N\|_{L^2} < \infty.
    \end{align*}
    This shows that $F(t)$ is bounded in every isotropic Sobolev space (cf.~\cite{DGW} for the definition) and this implies that $F(t) \in \calc^{-\infty}(\RR^d)$.
\end{proof}

\end{document}